\documentclass[12pt]{amsart}
\usepackage[margin=1.25in]{geometry}
\usepackage{latexsym}
\usepackage{amsmath}
\usepackage{amssymb}
\usepackage{amsthm}
\usepackage{stmaryrd}
\usepackage{amscd}
\usepackage{enumerate} 
\usepackage{amssymb} 
\usepackage{mathrsfs}
\usepackage[all]{xy}
\usepackage{color}
\usepackage{graphicx}
\usepackage{mathtools}
\usepackage{comment}
\usepackage{multirow}
\usepackage[pagebackref]{hyperref}
\hypersetup{colorlinks=true}

\makeatletter
  
  \@addtoreset{equation}{thm}
  \makeatother

\title{Endomorphisms of varieties
and Bott vanishing}
\author{Tatsuro Kawakami \and Burt Totaro}
\email{tatsurokawakami0@gmail.com}
\address{Department of Mathematics, Graduate School of Science, Kyoto University, Kyoto 606-8502, Japan}
\email{totaro@math.ucla.edu}
\address{UCLA Mathematics Department, Box 951555, Los Angeles,
CA, 90095-1555, U.S.}

\def\phi{\varphi}
\def\epsilon{\varepsilon}
\def\tilde{\widetilde}

\def\sing{\text{sing}}

\def\Spec{\operatorname{Spec}}

\def\Supp{\operatorname{Supp}}

\def\Pic{\operatorname{Pic}}

\def\Im{\operatorname{im}}
\def\Ker{\operatorname{ker}}

\def\rank{\operatorname{rank}}

\def\Cl{\operatorname{Cl}}
\def\Gr{\operatorname{Gr}}
\def\LGr{\operatorname{LGr}}
\def\OGr{\operatorname{OGr}}

\newcommand{\Q}{\mathbb{Q}}

\newcommand{\Z}{\mathbb{Z}}
\newcommand{\PP}{\mathbb{P}}
\newcommand{\FF}{\mathbb{F}}

\newcommand{\sO}{\mathcal{O}}
\newcommand{\sHom}{\mathcal{H}\! \mathit{om}}

\theoremstyle{plain}
\newtheorem{thm}{Theorem}[section] 

\newtheorem{prop}[thm]{Proposition}
\newtheorem{conj}[thm]{Conjecture}

\newtheorem{lem}[thm]{Lemma}
\theoremstyle{definition} 
\newtheorem{defn}[thm]{Definition}
\newtheorem{conv}[thm]{Convention}

\theoremstyle{remark}
\newtheorem{rem}[thm]{Remark}
\newtheorem{ques}[thm]{Question}

\newtheorem{defn and notation}[thm]{Definition and Notation}

\theoremstyle{plain}
\newtheorem{theo}{Theorem}


\begin{document}
\tolerance = 9999

\begin{abstract}
We show that a projective variety with an int-amplified
endomorphism of degree invertible in the base field
satisfies Bott vanishing. This is a new way
to analyze which varieties have nontrivial endomorphisms.
In particular, we extend some classification
results on varieties admitting endomorphisms
(for Fano threefolds
of Picard number one and several other cases)
to any characteristic. The classification results
in characteristic zero are due to
Amerik--Rovinsky--Van de Ven, Hwang--Mok, Paranjape--Srinivas,
Beauville, and Shao--Zhong. Our method also bounds
the degree of morphisms into a given variety.
Finally, we relate
endomorphisms to global $F$-regularity.
\end{abstract}

\maketitle
\markboth{Tatsuro Kawakami and Burt Totaro}{Endomorphisms of varieties and Bott vanishing}


\section{Introduction}

There is a long-standing conjecture about smooth Fano varieties admitting non-invertible surjective endomorphisms.

\begin{conj}\label{Introconj}
Let $X$ be a smooth Fano variety of Picard number 1
over an algebraically closed field of characteristic zero.
Suppose that $X$ admits a non-invertible surjective endomorphism. Then $X$ is isomorphic to projective space.
\end{conj}
Conjecture \ref{Introconj} has been proved when 
\begin{itemize}
       \item[\textup{(1)}] $\dim\,X=3$ \cite{Amerik-maps,ARV,Jun-Muk--Mok},
       \item[\textup{(2)}] $\dim\,X=4$ and $X$ has Fano index greater than 1 \cite{Shao-Zhong},
       \item[\textup{(3)}] $X$ is a hypersurface \cite{Paranjape-Srinivas,Beauville}, or
       \item[\textup{(4)}] $X$ is a homogeneous space
       \cite{Paranjape-Srinivas}.
\end{itemize}

The aim of this paper is to give a new approach to this problem and to generalize cases (1), (2), and (3) above
to arbitrary characteristic.

\begin{theo}\label{Introthm:rank one}
   Let $X$ be a smooth projective variety over an algebraically closed field $k$. Assume that $X$
   admits an endomorphism whose degree is greater
   than 1 and invertible in $k$.
   Suppose that one of the following holds.
   \begin{itemize}
       \item[\textup{(1)}] $X$ is a smooth Fano threefold of Picard number 1.
       \item[\textup{(2)}] $X$ is a smooth Fano fourfold
       of Picard number 1 and Fano index greater than 1.
       \item[\textup{(3)}] $X$ is a hypersurface
       of dimension at least 3.
   \end{itemize}
   Then $X$ is isomorphic to projective space.
\end{theo}

Our method also gives information on morphisms
other than endomorphisms. The following
result was known in characteristic zero
in cases (1) and (2) \cite{Amerik-maps}, \cite[Theorem 0.2]{ARV},
\cite[Theorem 2]{Jun-Muk--Mok}, 
\cite[Theorem 1.5]{Shao-Zhong},
and for quadrics in characteristic zero
in case (3)
\cite[Theorem]{Amerik-quadric}. Our proof
is short and valid in arbitrary characteristic.

\begin{theo}\label{Introthm:degree}
Let $X$ be one of the varieties in Theorem
\ref{Introthm:rank one}. Let $Y$ be a smooth projective
variety over $k$
of the same dimension that also has Picard number 1.
If $X$ is not isomorphic to projective space,
then there is an upper bound on the degrees
of all morphisms $Y\to X$ that have degree
invertible in $k$.
\end{theo}

The following assertion is a key ingredient for Theorem \ref{Introthm:rank one}. An endomorphism
$f\colon X\to X$ is said to be {\em int-amplified} if there is an ample Cartier divisor $H$ on $X$ such that $f^*H-H$ is ample \cite{Meng-building,Meng-Zhang-normal}. 

\begin{theo}\label{Introtheo:Bott}
Let $X$ be a normal projective variety over
a perfect field $k$.
Suppose that $X$ admits an int-amplified endomorphism
whose degree is invertible in $k$.
Then $X$ satisfies Bott vanishing for ample Weil divisors.
That is,
\[
H^i(X,\Omega^{[j]}_X(A))=0
\]
for every $i>0$, $j\geq 0$, and $A$ an ample Weil divisor.
\end{theo}

\begin{rem}
The assumption ``int-amplified'' is weaker than some
related conditions
on endomorphisms, such as {\em polarized},
meaning that there is an ample Cartier divisor $H$
with $f^*H\sim qH$ for some integer $q\geq 2$. For example,
the endomorphism $f(x,y)=(x^2,y^3)$ of $\PP^1\times \PP^1$
is int-amplified, but no positive iterate of $f$
is polarized. On the other hand, for a variety
with Picard group $\Z$,
every endomorphism of degree greater than 1
is polarized and hence int-amplified.
\end{rem}

\begin{rem}
    A smooth Fano variety that satisfies Bott vanishing
    is rigid, since $H^1(X,TX)
    =H^1(X,\Omega^{d-1}_X(-K_X))=0$, where $d$
    is the dimension of $X$.
    So Theorem \ref{Introtheo:Bott} implies
    that only finitely many smooth complex
    Fano varieties in each dimension admit
    an int-amplified endomorphism.
\end{rem}

\begin{rem}
    In proving Theorem \ref{Introtheo:Bott} for singular
    varieties, we develop some interesting tools.
    In particular, we prove the finiteness of flat cohomology
    $H^1(X,\mu_p)$ for smooth varieties $X$ over
    an algebraically closed field of characteristic $p$
    (Lemma \ref{lem:finiteness}).
\end{rem}

Finally, we show that a Fano variety with a suitable
endomorphism is well-behaved in characteristic $p$:

\begin{theo}\label{Introtheo:global}
Let $X$ be a Fano variety over a perfect field $k$
of characteristic $p>0$. Suppose that $X$ admits
an int-amplified endomorphism of degree prime to $p$.
If $X$ is strongly $F$-regular (for example, smooth),
then it is globally $F$-regular.
\end{theo}

It is known that the mod $p$ reductions of a klt Fano
variety in characteristic zero are globally $F$-regular
for sufficiently large primes $p$ (\cite[Theorem 1.2]{SS10}). The point of
Theorem \ref{Introtheo:global} is that it holds
even if $p$ is small. In this respect,
Fano varieties with a suitably nontrivial endomorphism
behave well,
somewhat like toric varieties.
For example, Petrov
showed that the Hodge spectral sequence degenerates
for all smooth projective varieties that are
globally $F$-split (which follows from
globally $F$-regular) \cite[Corollary 2.7.6]{Bhatt}.

For hypersurfaces of dimension at least 3 and degree
at least 2, it is straightforward to see that Bott vanishing fails. This implies Theorem \ref{Introthm:rank one}(3) by Theorem \ref{Introtheo:Bott}.

The proof of Theorem \ref{Introthm:rank one} is similar
for cases (1) and (2). To describe case (1):
we show that projective space is the only smooth Fano threefold of Picard number 1 that satisfies Bott vanishing (see also Remark \ref{rem:Bott by Totaro}).
In characteristic zero, this is an easy consequence of the classification \cite{Algebraic-Geometry-V}.
So assume that the characteristic $p$ is greater than 0.
Since we do not have such a complete classification in this case, we lift $X$ to characteristic zero.
By Theorem \ref{Introtheo:Bott},
we can take a lift $\tilde{X}_{\overline{K}}$ of $X$, which is a smooth Fano threefold of Picard number 1 over an algebraically closed field $\overline{K}$ of characteristic zero.
However, lifting endomorphisms is difficult in general, and therefore we prove that $\tilde{X}_{\overline{K}}$ satisfies
Bott vanishing instead. 
Then $\tilde{X}_{\overline{K}}\cong \PP^3_{\overline{K}}$ by the argument in characteristic zero.
Finally, observing that the Fano indices are preserved by lifting, we conclude that $X\cong\PP_k^3$. 

\begin{rem}\label{rem:Bott by Totaro}
     The paper \cite{Totaro} determines which smooth Fano
     threefolds in characteristic zero satisfy Bott vanishing
     \cite[Theorem 0.1]{Totaro}.    
    For smooth Fano threefolds in positive characteristic
    that are constructed in the same way as smooth Fano
    threefolds in characteristic zero, that paper also
    determines
    which ones satisfy Bott vanishing. Also, after
    this paper appeared on the arXiv, Tanaka
    showed that the classification of smooth Fano threefolds
    takes essentially the same form in any characteristic
    \cite[Theorem 1.1]{Tanaka-fano4}. The only question
    left open by Tanaka is whether there is a Fano 3-fold
    with Picard number 1, Fano index 1, and genus
    $g=11$ (that is, $(-K_X)^3=2g-2=20$) in some characteristic $p$, although this may be resolved
    soon. Independent of Tanaka's results,
    we show unconditionally that projective space
    is the only smooth Fano threefold of Picard number 1
    in any characteristic that satisfies Bott vanishing
    (Proposition \ref{prop:rank one}).
\end{rem}

\subsection{Related results}
\subsubsection{Two-dimensional case in positive characteristic.}
Nakayama showed that a smooth projective rational
surface in characteristic $p$ that admits
an endomorphism whose degree is greater than 1
and prime to $p$ must be toric
\cite[Proposition 4.4]{Nakayama(endomorphism)}.
On the other hand, he found a smooth rational surface that admits a separable polarized endomorphism but is not toric \cite[Example 4.5]{Nakayama(endomorphism)}.

\subsubsection{Failure of Bott vanishing for separable
polarized endomorphisms}
Answering a question in the first version of this paper,
we give an example to show
that Theorem \ref{Introtheo:Bott} fails
(that is, Bott vanishing fails)
if the assumption that the endomorphism $f$ has degree
invertible in $k$
is weakened to the assumption that $f$ is separable
(Proposition \ref{prop:separable}).
We do not know whether Theorems \ref{Introthm:rank one}
and \ref{Introthm:degree}
hold for $f$ separable rather than for $f$ of degree
invertible in $k$.

\subsubsection{Three-dimensional case in arbitrary characteristic.}
Normal projective $\Q$-Gorenstein threefolds that admit polarized endomorphisms in arbitrary characteristic were studied in detail by Cascini--Meng--Zhang
\cite[Theorem 1.8]{Cascini-Meng-Zhang}.
In particular, they proved that a smooth rationally chain connected threefold in characteristic $p$ admitting a polarized endomorphism $f$ has an $f$-equivariant
minimal model program if $p>5$ and the degree of the Galois closure $f^{\mathrm{Gal}}$ of $f$ is prime to $p$.

\subsubsection{Fano threefolds of arbitrary Picard number in characteristic zero.}
Meng--Zhang--Zhong \cite{Meng-Zhang-Zhong} proved that a smooth Fano threefold over an algebraically closed field of characteristic zero that admits an int-amplified endomorphism is toric.
Therefore, it is natural to ask the following question.
After the first version of this paper, a positive answer
was given in \cite[Theorem 6.1]{Totaro-endo}.
\begin{ques}
Let $X$ be a smooth Fano threefold over an algebraically closed field of characteristic $p>0$.
Suppose that $X$ admits an int-amplified endomorphism whose degree is prime to $p$.
Is $X$ toric?
\end{ques}

\subsection{Notation and terminology}
Unless otherwise mentioned, $k$ is an algebraically closed field of characteristic $p\geq0$ and a variety is defined over $k$.

A morphism of varieties $f\colon Y\to X$ over $k$
is {\it separable }if it is dominant
and $k(Y)$ is a separable field extension
of $k(X)$. Equivalently (for $k$ algebraically closed),
the derivative of $f$ is surjective at some smooth point
of $Y$.

For a normal variety $X$ over a field $k$ and
$i\geq 0$, we write $\Omega^i_X$
for $\Omega^i_{X/k}$. The sheaf
of {\em reflexive differentials} $\Omega^{[i]}_X$
is the double dual $(\Omega^i_X)^{**}$. More generally,
for a Weil divisor $D$ on $X$, we write $\Omega^{[i]}_X(D)$
for the reflexive sheaf $(\Omega^i_X\otimes \sO_X(D))^{**}$.
If $X$ is smooth over $k$, then $\sO_X(D)$ is a line bundle
and $\Omega^{[i]}_X(D)$
is just the tensor product $\Omega^i_X\otimes \sO_X(D)$.

\section{Bott vanishing and endomorphisms}
In this section,
we prove Theorem \ref{Introtheo:Bott}, relating
Bott vanishing with endomorphisms.
We also
prove a general relation between Bott vanishing
and morphisms into a given variety,
not just endomorphisms (Proposition
\ref{prop:degree bound}).

\subsection{Finiteness of flat cohomology}

The following lemma may be known, but we could not find
a reference in this generality.

\begin{lem}\label{lem:finiteness}
Let $X$ be a smooth variety over an algebraically closed
field $k$, and let $s$ be a positive integer.
Then the flat cohomology group $H^1(X,\mu_s)$ is finite.
\end{lem}

\begin{proof}
The problem reduces to the case where $s$ is prime.
If $s$ is invertible in $k$, then this finiteness holds
for cohomology in all degrees
\cite[Corollary VI.5.5]{Milne-etale}.
So we can assume that
$k$ has characteristic $p>0$ and $s=p$. (The result here is
special to $H^1$. Indeed, a supersingular K3 surface $X$
over $k$ has $H^2(X,\mu_p)$ infinite, containing
the additive group $k$
\cite[Proposition 4.2]{Artin-supersingular}.)
It is straightforward
to see that $H^1(X,\mu_p)$ injects into $H^1(k(X),\mu_p)=
k(X)^*/(k(X)^*)^p$ \cite[Lemma 3.9]{Keller}.

Let $\overline{X}$ be a normal compactification of $X$.
By de Jong, there is a separable alteration
$f\colon \overline{Y}\to\overline{X}$
\cite[Theorem 4.1]{deJong}.
That is, $f$ is generically
\'etale and $\overline{Y}$ is a smooth projective variety
over $k$.
Let $Y$ be the inverse
image of $X$ in $\overline{Y}$. Here $k(Y)$ is a finite separable
extension of $k(X)$, and so $k(X)^*/(k(X)^*)^p$ injects
into $k(Y)^*/(k(Y)^*)^p$. By the previous paragraph,
it follows that $H^1(X,\mu_p)$ injects into $H^1(Y,\mu_p)$.

So it suffices to show that $H^1(Y,\mu_p)$ is finite.
Consider the Kummer sequence $\sO(Y)^*/(\sO(Y)^*)^p\to H^1(Y,\mu_p)
\to \Pic(Y)[p]$. The group of units $\sO(Y)^*$ is an extension of a finitely generated abelian group by $k^*$
\cite[Lemme 1]{Kahn},
and so $\sO(Y)^*/(\sO(Y)^*)^p$ is finite.
So it suffices to show that $\Pic(Y)[p]$ is finite.

Since $\overline{Y}$ is smooth, we have $\Pic(Y)=\Pic(\overline{Y})/M$,
where $M$ is the subgroup generated by the codimension-1
subvarieties of $\overline{Y}$ contained
in $\overline{Y}-Y$. In particular, the abelian group $M$
is finitely generated. Applying the snake lemma to the map of exact sequences
\[
\xymatrix{
 0 \ar[r] &M\ar[r]\ar[d]^{ p} & \Pic(\overline{Y})\ar[r] \ar[d]^{ p} &\Pic(Y)\ar[r]\ar[d]^{ p} & 0\\
 0 \ar[r]& M\ar[r] & \Pic(\overline{Y})\ar[r]  &\Pic(Y)\ar[r] & 0,
}
\]
we obtain an exact sequence
$\Pic(\overline{Y})[p]\to \Pic(Y)[p]\to M/p$. So $\Pic(Y)[p]$ is finite
if $\Pic(\overline{Y})[p]$ is finite. That is immediate from
the structure of the Picard group of a smooth projective
variety:
$\Pic(\overline{Y})$ is an extension of a finitely generated abelian
group $NS(\overline{Y})$ by the $k$-points of an abelian variety
\cite[Th\'eor\`eme XIII.5.1]{SGA6}. Lemma \ref{lem:finiteness}
is proved.
\end{proof}

\subsection{Bott vanishing}

In this subsection, we prove Theorem \ref{Introtheo:Bott}.
We are generalizing Fujino's proof of Bott vanishing
for toric varieties, based on the existence of suitable
endomorphisms, beyond the toric setting. 

\begin{defn}\label{def:Bott vanishing}
   Let $X$ be a smooth projective variety over a field.
   We say that $X$ satisfies \textit{Bott vanishing }if we have
   \[
    H^i(X,\Omega_X^j(A))=0
   \]
  for every $i>0$, $j\geq 0$, and $A$ an ample Cartier
    divisor.
\end{defn}

For singular varieties, we consider the following
strong version of Bott vanishing. A Weil divisor
(with integer coefficients) is called {\it ample }if
some positive multiple is an ample Cartier divisor.
Likewise for {\it nef}.

\begin{defn}\label{def:Bott-singular}
Let $X$ be a normal projective variety over a field.
We say that $X$ satisfies \textit{Bott vanishing
for ample Weil divisors }if we have
\[
H^i(X,\Omega_X^{[j]}(A))=0
\]
for every $i>0$, $j\geq 0$, and $A$ an ample Weil divisor.
\end{defn}

\begin{rem}
    All projective toric varieties satisfy Bott vanishing
    for ample Weil divisors, by Fujino
    \cite[Proposition 3.2]{Fujino-toric}.
\end{rem}

\begin{proof}[Proof of Theorem \ref{Introtheo:Bott}]
    Since $k$ is perfect, the normal variety $X$ over $k$
    is geometrically normal \cite[Tag 038O]{stacks-project}.
    Replacing $k$ with its algebraic closure, we may assume that $k$ is algebraically closed.
    For clarity, we first prove the theorem
    for $X$ smooth. (That is enough for the applications
    in this paper.)
    So let $X$ be a smooth projective variety
    over an algebraically closed field $k$
    with an endomorphism $f$ and an ample Cartier divisor $H$ on $X$ such
    that $f^*H-H$ is ample. In particular, $f^*H$
    is ample, and so $f$ does not contract any curves.
    Therefore, $f\colon X\to X$ is finite. We assume
    that the degree of $f$ is invertible in $k$.
    
    Let $A$ be any ample Cartier divisor on $X$.
    Fix $i>0$ and $j\geq 0$. We want to show that
    \[
    H^i(X, \Omega_X^j(A))=0.
    \]
    We will use Fujita's vanishing theorem \cite[Theorem 1]{Fujita}:

\begin{thm}\label{theo:fujita}
    Let $X$ be a projective scheme over a field,
    $H$ an ample Cartier divisor on $X$, and $E$
    a coherent sheaf on $X$. Then there is a positive
    integer $m$ such that $H^i(X,E\otimes \sO_X(mH+D))=0$
for every $i>0$ and every nef Cartier divisor $D$ on $X$.
\end{thm}
    
    Since $f$
    is finite and $X$ is smooth over $k$,
    there is a {\em trace map }$\tau_f\colon f_{*}\Omega^j_X\to \Omega^j_X$ such that the composition
    \[
    \Omega^j_X\overset{f^{*}}{\to} f_{*}\Omega^j_X\overset{\tau_f}{\to} \Omega^j_X
    \] 
    is multiplication by $\deg(f)$,
    by Garel and Kunz \cite{Garel},
    \cite[section 16]{Kunz}, \cite[Tag 0FLC]{stacks-project}.
    Thus $\frac{1}{\deg(f)}\tau_f$ gives a splitting of the pullback $f^{*}\colon \Omega^j_X\hookrightarrow f_{*}\Omega^j_X$. Taking the pushforward by $f$,
    we obtain a split injective map
    $f_*\Omega^j_X\hookrightarrow (f^2)_*\Omega^j_X$,
    and thus a split injective map
    $\Omega^j_X\hookrightarrow (f^2)_*\Omega^j_X$.
    Repeating this procedure,
    for every positive integer $e$,
    $(f^{e})^{*}\colon \Omega^j_X\hookrightarrow
    (f^{e})_{*}\Omega^j_X$ splits. Tensoring with $\sO_X(A)$,
    we have a split injective map
    \[
    \Omega^j_X(A)\to (f^{e})_{*}(\Omega^j_X((f^{e})^{*}A)).
    \]
    Taking cohomology (and using that $f^e$ is finite), we have a split injective map
    \[
    H^i(X, \Omega^j_X(A))\hookrightarrow
    H^i(X, \Omega^j_X((f^{e})^{*}A)).
    \]
    So it suffices to find an $e\geq 1$
    such that the right hand side is zero.
    Let $m$ be a positive integer associated
    to the given ample Cartier divisor $H$ (with $f^*H-H$ ample) and the coherent sheaf $E=\Omega^i_X$, in Fujita
    vanishing (Theorem \ref{theo:fujita}). Then it suffices
    to find an $e\geq 1$ such that $(f^e)^*A-mH$ is nef.

    Since $f^*H-H$ is ample, there is a rational number
    $c>1$ such that $f^*H-cH$ is ample. Here $f^*$
    takes nef divisors to nef divisors. So, for every
    $e\geq 1$, $(f^e)^*H-c^eH$ is nef.

    Since $A$ is ample, there is a rational number $u>0$
    such that $A-uH$ is ample. Using again that $f^*$
    takes nef divisors to nef divisors, we find that
    for every $e\geq 1$,
    $(f^e)^*A-u(f^e)^*H$ is nef. It follows that
    $(f^e)^*A-uc^eH$ is nef. There is a positive integer
    $e$ such that $uc^e\geq m$. Then $(f^e)^*A-mH$
    is nef, as we want. Bott vanishing is proved.

    More generally, assume that $X$ is
    normal (rather than smooth).
    Since $f$ is a finite morphism
    between normal varieties,
    reflexive differentials $\Omega^{[j]}_X=
    (\Omega^j_X)^{**}$ pull back under $f$.
    (We can pull differential forms back outside
    $X^{\sing}\cup f(X^{\sing})$, and that gives
    a pullback map on reflexive differentials
    since the complement
    has codimension at least 2.)
    That is, we have a pullback map
    $\Omega^{[j]}_X\to f_*\Omega^{[j]}_X$. Likewise,
    the trace map $\tau_{f}
    \colon f_*\Omega^{[j]}_X\to \Omega^{[j]}_X$ is defined
    outside $X^{\sing}\cup f(X^{\sing})$, so it extends
    to all of $X$ since the sheaf $\Omega^{[j]}_X$
    is reflexive. Given this,
    the proof above works without change for an ample
    Cartier divisor $A$.

    Finally, to prove the full theorem, assume
    that $X$ is normal
    and $A$ is an ample Weil divisor. For this, we need
    a version of Fujita vanishing for $\Q$-Cartier Weil divisors:

    \begin{lem}\label{lem:FujitaWeil}
    Let $X$ be a normal projective variety
    over an algebraically closed field $k$,
    $H$ an ample Cartier divisor
    on $X$, $E$ a reflexive sheaf on $X$, and $s$ a positive
    integer.
    Then there is a positive integer $m$ such that
    $H^i(X,E(mH+D))=0$ for every
    $i>0$ and every nef Weil divisor $D$ such that $sD$
    is Cartier.
    \end{lem}

    \begin{proof}
        Since $X$ is a normal projective variety
        over $k$, the Picard group $\Pic(X)$ is an
        extension of a finitely generated abelian group
        by the $k$-points of an abelian variety
        \cite[Th\'eor\`eme XIII.5.1]{SGA6},
        \cite[Theorem 5.4]{Kleiman}. Therefore, $\Pic(X)/s$
        is finite.
        Let $U$ be the smooth locus of $X$, so that
        the divisor class group $\Cl(X)$ is isomorphic
        to $\Pic(U)$. By Lemma \ref{lem:finiteness},
        $H^1(U,\mu_s)$ is finite.
        By the Kummer sequence $H^1(U,\mu_s)
        \to H^1(U,G_m)\xrightarrow[]{s} H^1(U,G_m)$, the $s$-torsion
        subgroup $\Cl(X)[s]=\Pic(U)[s]$ is also finite. By tensoring the exact sequence $0\to \Pic(X)\to \Cl(X)
        \to \Cl(X)/\Pic(X)\to 0$ over $\Z$ with $\Z/s$,
        we have an exact sequence $\Cl(X)[s]\to (\Cl(X)/\Pic(X))
        [s]\to \Pic(X)/s$. So $(\Cl(X)/\Pic(X))[s]$ is finite.

        Let $D_1,\ldots,D_r$ be Weil divisors with $sD_j$
        Cartier that represent every element of the group
        $(\Cl(X)/\Pic(X))[s]$. By subtracting a suitable multiple
        of $H$ from each $D_j$, we can assume that $-D_j$
        is nef for each $j$. Then, for every nef Weil divisor
        $D$ on $X$ with $sD$ Cartier, there is a $1\leq j
        \leq r$ such that $D-D_j$ is Cartier. Apply Fujita's
        theorem (Theorem \ref{theo:fujita}) to the coherent
        sheaf $\oplus_{j=1}^rE(D_j)$. (By definition,
        $E(D_j)$ means
        the reflexive sheaf $(E\otimes \sO_X(D_j))^{**}$.)
        This gives that there is a positive integer $m$
        such that $H^i(X,E(mH+D_j+N))=0$
        for every $i>0$, every $1\leq j\leq r$,
        and every nef Cartier divisor $N$. Since $-D_j$
        is nef for each $j$, it follows that for every nef
        Weil divisor $D$ with $sD$ Cartier, we have
        $H^i(X,E(mH+D))=0$
        for every $i>0$. Lemma \ref{lem:FujitaWeil}
        is proved.
    \end{proof}

    We now complete the proof of Theorem \ref{Introtheo:Bott}
    for $X$ normal over an algebraically closed field $k$
    and an ample Weil divisor $A$. This follows by the
    argument above (where $A$ is an ample Cartier divisor),
    using Lemma \ref{lem:FujitaWeil} in place of Fujita's
    theorem. There is a positive
    integer $s$ such that $sA$
    is Cartier. Since $f$ is a finite morphism
    between normal varieties, $(f^e)^*A$ is a Weil divisor
    (with integer coefficients) for every positive integer
    $e$, and $s(f^e)^*A$ is Cartier. This is what we need
    in order to apply Lemma \ref{lem:FujitaWeil}.
    Theorem \ref{Introtheo:Bott} is proved.
    \end{proof}

\subsection{Bounding the degree of morphisms}

\begin{prop}\label{prop:degree bound}
Let $X$ and $Y$ be normal projective varieties
of the same dimension over a field $k$, and suppose
that both have Picard number 1. If there are morphisms
from $Y$ to $X$ with arbitrarily large degree
such that the degree is invertible in $k$,
then $X$ must satisfy Bott vanishing
for ample Weil divisors.
\end{prop}

\begin{proof}
Let $f\colon Y\to X$ be a morphism whose degree
is invertible in $k$.
Let $H$ be an ample Cartier divisor on $X$. Then $f^*H$
has positive degree on some curve, hence is ample
since $Y$ has Picard number 1. It follows that
$f\colon Y\to X$ is finite.

    Let $A$ be an ample Weil divisor on $X$,
    and fix $i>0$ and $j\geq 0$. We need to prove that
    \[
    H^i(X,\Omega_X^{[j]}(A))=0.
    \]
    Since $\deg(f)$ is invertible in $k$,
    we have a split injection
    \[
    \Omega_X^{[j]}(A)\hookrightarrow
    f_*(\Omega_Y^{[j]}(f^{*}A)),
    \]
    as in the proof
    of Theorem \ref{Introtheo:Bott}.
    Taking cohomology (and using that $f$ is finite),
    we have a split injection
    \[
    H^i(X,\Omega_X^{[j]}(A))\hookrightarrow
    H^i(Y,\Omega_Y^{j}(f^{*}A)).
    \]

    If there are morphisms $f\colon Y\to X$
    with arbitrarily large degree such that
    the degree is invertible in $k$, then $f^*A$
    becomes arbitrarily large in the ample cone
    of $Y$ (here just one ray). So Fujita vanishing
    for Weil divisors
    (Lemma \ref{lem:FujitaWeil}) gives that,
    for $f$ of sufficiently large degree,
    $H^i(Y,\Omega_Y^{[j]}(f^*A))=0$.
    By the previous paragraph, it follows that
    $H^i(X,\Omega_X^{[j]}(A))=0$.
    Bott vanishing is proved.
\end{proof}

\section{Bott vanishing for specific classes
of varieties}
\subsection{Hypersurfaces}

In this subsection, we prove that projective space is the only smooth hypersurface of dimension at least 3
that satisfies Bott vanishing.

\begin{lem}\label{lem:Calabi-Yau hypersurface}
    Let $X\subset \PP^{d+1}$ be a smooth hypersurface
    of degree $d+n$ over a field.
    Suppose that $d>1$ and $n>0$.
    Then $H^{d-1}(X, \Omega_X^{1}(n))\neq 0$, and in particular, $X$ does not satisfy Bott vanishing.
\end{lem}
\begin{proof}
   By adjunction, the dualizing sheaf $\omega_X$
   is isomorphic to $\sO_X(n-2)$. So Serre duality
   gives that
   $H^d(X, \sO_X(-d))\cong H^0(X, \sO_X(d+n-2))^*\neq 0$.
   By the Euler sequence
   \[0\to \sO_{\PP^{d+1}}\to
   \sO_{\PP^{d+1}}(1)^{\oplus d+2}\to T\PP^{d+1}\to 0
   \]
   we have $H^0(\PP^{d+1},T\PP^{d+1}(-2))=0$.
   By the exact sequence
   \[
   0\to \Omega^1_{\PP^{d+1}}(-d) \to \Omega^1_{\PP^{d+1}}(n) \to \Omega^1_{\PP^{d+1}}(n)|_X\to 0,
   \]
   we have $H^{d}(X, \Omega^1_{\PP^{d+1}}(n)|_X)=0$.
   Here we used that $H^{d+1}(\PP^{d+1}, \Omega^1_{\PP^{d+1}}(-d))\cong H^{0}(\PP^{d+1}, T\PP^{d+1}(-2))^*=0$
   and Bott vanishing on $\PP^{d+1}$.
   By the exact sequence
   \[
   0\to \sO_{X}(-d) \to \Omega^1_{\PP^{d+1}}(n)|_X \to \Omega^1_{X}(n)\to 0, 
   \]
   we have $H^{d-1}(X, \Omega_X^{1}(n))\neq 0$, as desired.
\end{proof}

\begin{prop}\label{prop:Bott on hypersurface}
    Let $X\subset \PP^{d+1}$ be a smooth hypersurface
    of dimension at least 2
    over an algebraically closed field.
    If $X$ satisfies Bott vanishing, then $X\cong \PP^{d}$
    or $X$ is the quadric surface $\PP^1\times \PP^1$.
\end{prop}
\begin{proof}
    A Fano variety that satisfies Bott vanishing
    is rigid, since $H^1(X,TX)=H^1(X,\Omega^{d-1}_X(-K_X))=0$,
    where $d=\dim(X)$.
    That excludes all Fano hypersurfaces of degree at least 3
    \cite[equation 5.21]{Kodaira-book}.
    Bott vanishing also fails for quadrics
    of dimension at least 3 \cite[Example 3.2.6]{AWZ}.
    Finally, Lemma
    \ref{lem:Calabi-Yau hypersurface} shows that Bott
    vanishing fails for all non-Fano hypersurfaces
    of dimension at least $2$.
    (This is mostly relevant for the Calabi-Yau case.
    Indeed, Bott vanishing fails
    for all varieties of positive dimension $d$ with ample
    canonical class, since $H^d(X,\omega_X)=k\neq 0$.)
    \end{proof}

\subsection{Fano threefolds}

In this subsection, we prove that projective space is the only Fano threefold of Picard number 1 that satisfies
Bott vanishing, in any characteristic.
Throughout this subsection, we use the following convention.
\begin{conv} 
Let $k$ be an algebraically closed field of positive
characteristic.
We denote $W(k)$ the ring of Witt vectors and $K$ the field
of fractions of $W(k)$.
For a proper scheme $X$ over $k$, we say that a scheme
$\tilde{X}$ over $W(k)$ is a \textit{lift of $X$} if
$\tilde{X}\otimes_{W(k)} k \cong X$ and $\tilde{X}$
is flat and proper over $W(k)$.
For a lift $\tilde{X}$ and a Cartier divisor $\tilde{A}$ on
$\tilde{X}$, we denote the geometric generic fiber of
$\tilde{X}\to \Spec\,W(k)$ by $\tilde{X}_{\overline{K}}$ and
the pullback of $\tilde{A}$ to $\tilde{X}_{\overline{K}}$ by
$\tilde{A}_{\overline{K}}$.

\end{conv}

\begin{defn}
    Let $X$ be a smooth Fano variety
    over an algebraically closed field.
    The \textit{Fano index} $r(X)\in \Z_{>0}$ of $X$ is the largest integer $n\in\Z_{>0}$ such that $-K_X\sim nA$ for some Cartier divisor $A$.
\end{defn}

\begin{lem}\label{lem:Fano index}
    Let $X$ be a smooth Fano variety with $d\coloneqq \dim\,X$.
    Suppose that $X$ satisfies Kodaira vanishing.
    If $r(X)\geq d+1$, then $X\cong \PP^{d}$.
\end{lem}
\begin{proof}
    The proof of \cite[Proposition 4]{Megyesi} works in any dimension, since we assume that $X$ satisfies Kodaira vanishing.
\end{proof}

\begin{prop}\label{prop:lift of Fano}
    Let $X$ be a smooth Fano variety
    over an algebraically closed field $k$ of positive characteristic.
    Suppose there exists a lift $\tilde{X}$ over $W(k)$ of $X$.
    In addition, assume that $H^1(X, \sO_X)=H^2(X, \sO_X)=0$.
    Then the following statements hold.
    \begin{enumerate}
        \item[\textup{(1)}] The specialization map $\mathrm{sp}\colon\Pic(\tilde{X}_{\overline{K}})\to \Pic(X)$ is an isomorphism of abelian groups.  
        \item[\textup{(2)}] $r(X)=r(\tilde{X}_{\overline{K}})$.
        \item[\textup{(3)}] Suppose that the Picard number of $X$ is 1. If $X$ satisfies Bott vanishing, then so does $\tilde{X}_{\overline{K}}$.
    \end{enumerate}
\end{prop}
\begin{rem}
    For a smooth Fano threefold $X$, Shepherd-Barron \cite[Corollary 1.5]{SB97} proved that $H^1(X, \sO_X)=H^2(X, \sO_X)=0$ (see also \cite[Corollary 3.7]{Kaw2} for the case where $p=2$ or $3$).
    Also, for a smooth Fano variety $X$ of any dimension
    that satisfies Bott vanishing (hence Kodaira
    vanishing), we have $H^i(X,\sO_X)=0$ for $i>0$. See
    \cite[Theorem 1.1, Proposition 6.3]{Gounelas-Javanpeykar}
    for what is known about the Fano index in families
    without assuming
    that $H^1(X, \sO_X)=H^2(X, \sO_X)=0$.
\end{rem}
\begin{proof}
    First, we prove (1). Since smoothness and ampleness are open properties in a flat proper family,
    the geometric generic fiber $\tilde{X}_{\overline{K}}$ is a smooth Fano variety and $\Pic(\tilde{X}_{\overline{K}})\cong \mathrm{NS}(\tilde{X}_{\overline{K}})$ is a free $\Z$-module, where $\mathrm{NS}(X)$ denotes the N\'eron-Severi group.
    Let $\mathrm{sp}\colon \Pic(\tilde{X}_{\overline{K}})\to \Pic(X)$ be the specialization map (see the proof of \cite[Proposition 3.3]{MP12} for the construction).
    Since $H^1(X,\sO_X)=0$, we have $\Pic(X)=\mathrm{NS}(X)$ by \cite[Theorem 9.5.11]{FAG}. The Picard group
    of the smooth Fano variety $\tilde{X}_{\overline{K}}$
    in characteristic zero is torsion-free;
    for a quick proof,
    see \cite[Introduction]{Fanelli-unusual}.
    Next, \cite[Proposition 3.6]{MP12} gives that
    the specialization $NS(\tilde{X}_{\overline{K}})\to
    NS(X)$ is injective. Since $\Pic(X_{\overline{K}})
    =NS(X_{\overline{K}})$ and $\Pic(X)=NS(X)$ in our case,
    it follows that
    $\mathrm{sp}\colon \Pic(\tilde{X}_{\overline{K}})\to \Pic(X)$
    is injective.
    Since $H^2(X, \sO_X)=0$, the specialization map
    $\mathrm{sp}$ is also surjective
    \cite[Corollary 8.5.6]{FAG}.
    Thus (1) holds.

    Next, we prove (2). By the definition of $r(\tilde{X}_{\overline{K}})$, 
    we can take an ample Cartier divisor $\overline{A}$ on $\tilde{X}_{\overline{K}}$ such that 
    $-K_{\tilde{X}_{\overline{K}}}\sim r(\tilde{X}_{\overline{K}})\overline{A}$.
    Then we have 
    \[
    -K_X\sim\mathrm{sp}(-K_{\tilde{X}_{\overline{K}}})\sim \mathrm{sp}(r(\tilde{X}_{\overline{K}})\overline{A})=r(\tilde{X}_{\overline{K}})\mathrm{sp}(\overline{A}),
    \]
    which shows that $r(\tilde{X}_{\overline{K}})\leq r(X)$.

    By the definition of $r(X)$, we can take an ample Cartier divisor $A$ on $X$ such that $-K_X\sim r(X)A$.
    Let $\tilde{A}\in\Pic(\tilde{X})$ be a lift of $A$.
    Then we obtain $-K_{\tilde{X}_{\overline{K}}}\sim r(X)\tilde{A}_{\overline{K}}$, which shows that $r(X)\leq r(\tilde{X}_{\overline{K}})$.
    Thus, (2) holds.

Finally, we prove (3).
    Take an ample Cartier divisor $\overline{A}$ on $\tilde{X}_{\overline{K}}$
    and fix $i>0$ and $j\geq 0$. We prove
    \[
    H^i(\tilde{X}_{\overline{K}}, \Omega_{\tilde{X}_{\overline{K}}}^j(\overline{A}))=0.
    \]
    Let $A\coloneqq \mathrm{sp}(\overline{A})$. 
    Then $\overline{A}\sim\tilde{A}_{\overline{K}}$ for a lift $\tilde{A}$ of $A$ by the argument in (1).
    Since $\overline{A}$ is ample, we can take $m\gg 0$ such that $h^0(\tilde{X}_{\overline{K}}, \sO_{\tilde{X}_{\overline{K}}}(m\tilde{A}_{\overline{K}}))>1$, and upper semi-continuity (\cite[Theorem III.12.8]{Har}) shows that $h^0(X, \sO_X(mA))>1$.
    Since the Picard number of $X$ is 1, it follows that $A$ is ample.
    By assumption, we have
    \[
    H^i(X, \Omega_X^j(A))=0,
    \]
    and upper semi-continuity shows the desired vanishing
    in characteristic zero. Thus, (3) holds.
\end{proof}

   We use Proposition \ref{prop:lift of Fano}
   to reduce the following Proposition
   to the case of characteristic zero. 

\begin{prop}\label{prop:rank one}
    Let $X$ be a smooth Fano threefold of Picard rank 1
    over an algebraically closed field $k$
    such that $X$ satisfies Bott vanishing.
    Then $X$ is isomorphic to projective space.
\end{prop}
\begin{proof}
    \textbf{Step 1} (characteristic zero case).
    Assume that $k$ has characteristic zero.
    By Bott vanishing, we have $H^i(X,TX)=H^i(X, \Omega_X^{2}(-K_X))=0$ for all $i>0$.
    By the Hirzebruch–Riemann–Roch theorem (cf.~\cite[Proof of Theorem 7.4]{AWZ2}), we have
    \[
    0\leq h^0(X, TX)=\chi(X, TX)=\frac{1}{2}(-K_X)^3-18+\rho(X)-h^1(X, \Omega^2_X), 
    \]
    where $\rho(X)$ is the Picard number of $X$.
    Then by \cite[Tables 12.2]{Algebraic-Geometry-V}, it follows that $X$ is isomorphic to the quintic
    del Pezzo threefold
    $V_5$ (a smooth codimension-3 linear section
    of $\Gr(2,5)\subset \PP^9$), the quadric threefold $Q$, or $\PP^3$.
    Since $V_5$ and $Q$ do not satisfy Bott vanishing by
    \cite[Lemma 7.10]{AWZ2} and 
    \cite[subsection 4.1]{BTLM} (or Proposition
    \ref{prop:Bott on hypersurface}), it follows that $X\cong \PP^3$.\\

    \textbf{Step 2} (positive characteristic case). 
    Assume that $k$ has positive characteristic.
    Since $X$ is Fano and satisfies Bott vanishing, we have 
    \[
    H^i(X, TX)=H^i(X, \sO_X)=0
    \]
    for all $i>0$. 
    By \cite[Theorem 9.5.11]{FAG}, we can take a lift $\tilde{X}$ over $W(k)$ of $X$, and the geometric generic fiber $\tilde{X}_{\overline{K}}$ is a smooth Fano threefold.
    By Proposition \ref{prop:lift of Fano} (1) and (3), the Picard number of $\tilde{X}_{\overline{K}}$ is 1 and $\tilde{X}_{\overline{K}}$ satisfies Bott vanishing.
    Thus, by Step 1, we have $\tilde{X}_{\overline{K}}\cong \PP^3_{\overline{K}}$. By Proposition \ref{prop:lift of Fano}(2), we have $r(X)=r(\tilde{X}_{\overline{K}})=4$, and we conclude that $X\cong \PP^3_k$ by Lemma
    \ref{lem:Fano index}.
\end{proof}

\subsection{Images of toric varieties}

Let $f\colon Y\to X$ be a morphism from a projective toric variety $Y$ onto a smooth projective variety $X$ of Picard number 1.
In characteristic zero, generalizing Lazarsfeld’s result
on images of projective space
\cite[Theorem 4.1]{Lazarsfeld1983}, Occhetta-Wi\'{s}niewski \cite[Theorem 1.1]{Occhetta-Wisniewski} proved that $X$ is isomorphic to projective space.

This result does not extend to characteristic $p>0$
in full generality.
For example, in characteristic 2, there is a finite purely inseparable morphism from $\PP^3$ onto a smooth quadric
threefold \cite[Proposition 2.5]{Ekedahl}.
However, Occhetta-Wi\'{s}niewski's proof does work
without change
for {\it separable }morphisms in positive characteristic.
That is:

\begin{thm}\label{thm:occhetta}
    Let $X$ be a smooth projective variety of Picard
    number 1 over an algebraically closed field.
    Let $Y$ be a proper toric variety. If there is
    a separable morphism $Y\to X$,
    then $X$ is isomorphic to projective space.
\end{thm}

For possible future use, let us show how
a special case of Theorem \ref{thm:occhetta}
follows from our arguments with Bott vanishing.
One can also
prove a version of Proposition \ref{prop:images} using
Achinger--Witaszek--Zdanowicz's results on images
of $F$-liftable varieties \cite[Theorem 4.4.1]{AWZ}.
For example, they showed that a smooth complex surface
that is an image of a proper toric variety
must be toric \cite[Theorems 2 and 3]{AWZ}.

\begin{prop}\label{prop:images}
    Let $X$ be a smooth projective threefold
    of Picard number 1 over an algebraically closed
    field $k$. Let $Y$ be a proper toric variety
    of the same dimension. If there is a morphism
    $Y\to X$ of degree invertible in $k$,
    then $X$ is isomorphic to projective space.
\end{prop}
\begin{proof}
    Let $Y\to Y'\to X$ be the Stein factorization.
    Replacing $Y$ by $Y'$, we may assume that $f$ is finite and $Y$ is a normal toric variety,
    by \cite[Proposition 2.7]{Tanaka-kv}.

    Let $A$ be an ample Cartier divisor on $X$,
    and fix $i>0$ and $j\geq 0$.
    By the same proof
    as in Proposition \ref{prop:degree bound}
    (pushing forward differential forms),
    we have a split injection
    \[
    H^i(X,\Omega_X^j(A))\hookrightarrow H^i(Y,\Omega^{[j]}_Y(f^{*}A)),
    \]
    where $\Omega^{[j]}_Y$ is the sheaf
    of reflexive differentials,
    $(\Omega_Y^j)^{**}$.
    Since $f^*A$ is ample, Bott vanishing
    for toric varieties gives that the group
    on the right is zero, and so the group
    on the left is zero. That is, $X$ satisfies
    Bott vanishing.

    Since $f$ is separable, we have $f^*(-K_X)=-K_Y+R$
    for an effective divisor $R$, the {\it ramification
    divisor }\cite[equation 2.41.2]{Kol13}.
    Since $-K_Y$ is big and the Picard number of $X$ is 1, it follows
    that $-K_X$ is ample.
    Therefore, $X\cong \PP^3$ by Proposition \ref{prop:rank one}.
\end{proof}

\subsection{Fano fourfolds of index greater than 1}

In this subsection, we prove that projective space
is the only Fano fourfold of Picard number 1
and Fano index greater than 1 that satisfies
Bott vanishing.

\begin{lem}\label{lem:hyperplane}
    Let $Y\subset \PP^N$ be a smooth
    projective variety of dimension at least 2
    over an algebraically
    closed field $k$ of characteristic zero,
    and let $X$ be a smooth hyperplane
    section. 
    Set $-K_Y=\sO_Y(b)$ and assume that
    $b\geq 2$.
    Then the following hold.
    \begin{itemize}
        \item[\textup{(1)}] We have
        \[
        \chi(X,TX)=\chi(Y,TY)-\chi(Y,TY(-1))-h^0(Y,\sO(1))+1.
        \]
        \item[\textup{(2)}] We have
        \[
        \chi(X,TX(-1))=\chi(Y,TY(-1))-\chi(Y,TY(-2))-1.
        \]
        \item[\textup{(3)}] For $2\leq a\leq b-1$, we have
    \[
    \chi(X,TX(-a))=\chi(Y,TY(-a))-\chi(Y,TY(-a-1)).
    \]
    \end{itemize}
    \end{lem}

    \begin{proof}
    We have exact sequences of coherent sheaves: 
    \begin{align*}
        &0\to TY(-1)\to TY\to TY|_X\to 0\,\,\,\text{and}\\
        &0\to TX\to TY|_X\to \sO_X(1)\to 0.
    \end{align*}
    For any integer $a$, it follows that
    \[\chi(X,TX(-a))=\chi(Y,TY(-a))-\chi(Y,TY(-a-1))
    -\chi(X,\sO_X(1-a)).\] Since $K_Y=\sO_Y(-b)$,
    adjunction gives that
    $K_X=(K_Y+X)|_X=\sO_X(1-b)$, and so $X$ is Fano.
    By Kodaira vanishing, we have $H^{>0}(X, \sO_X(1-a))=0$ for all $a\leq b-1$. Also, since $\sO_X(1)$ is ample,
    we have $H^0(X,\sO_X(1-a))=0$ for all $a\geq 2$,
    whereas $h^0(X,\sO_X(1-a))=1$ for $a=1$.
    Therefore, $\chi(X,\sO_X(1-a))$
    is zero for $2\leq a\leq b-1$, and it is 1
    for $a=1$, proving statements (2) and (3).
    For $a=0$, the exact sequence
    $0\to\sO_Y\to\sO_Y(1)\to\sO_X(1)\to 0$ shows
    that $\chi(X,\sO_X(1))=\chi(Y,\sO_Y(1))-1=
    h^0(Y,\sO_Y(1))-1$, proving (1).
    \end{proof}

\begin{prop}\label{prop:index two}
Let $X$ be a smooth Fano fourfold of Picard number 1
and Fano index greater than 1 over an algebraically
closed field $k$. If $X$ satisfies Bott vanishing,
then $X$ is isomorphic to projective space.
\end{prop}

\begin{proof}
    If $k$ has characteristic $p>0$, then
    (as in the proof of Proposition \ref{prop:rank one})
    $X$ lifts to characteristic 0, and the lift
    also satisfies Bott vanishing. So it suffices
    to prove the proposition for $k$ of characteristic zero.
    (That will imply that the lift has Fano index 5,
    so $X$ in characteristic $p$ has Fano index 5
    and satisfies Kodaira vanishing;
    so it is isomorphic to $\PP^4_k$.)

    So assume that $k$ has characteristic zero.
    Then the smooth Fano fourfolds of Picard number 1
    and index greater than 1 were classified
    by Fujita, Mukai, and Wilson
    \cite[Theorem 1.2]{Kuznetsov-Prokhorov},
    \cite[Theorem 2]{Mukai-Fano},
    \cite{Wilson}. The classification is listed
    in Table \ref{fourfold}, where the calculations
    of $\chi(X,TX)$ can be made
    using Lemma \ref{lem:hyperplane}. Since we assume
    that $X$ satisfies Bott vanishing,
    the tangent bundle $TX=\Omega^3_X(-K_X)$
    has zero cohomology in positive degrees,
    and so $\chi(X,TX)\geq 0$.
    (Here we used that on a smooth $d$-dimensional
    variety $X$, we have a dual pairing $\Omega^1_X\times
    \Omega^{d-1}_X\to \Omega^d_X=\sO(K_X)$, and so
    $TX\cong \Omega^{d-1}(-K_X)$.)
    By Table \ref{fourfold},
    $X$ is either
    $\PP^4$, the quadric 4-fold $Q$, a codimension-2
    linear section of the Grassmannian $\Gr(2,5)
    \subset \PP^9$,
    or a hyperplane section of the $G_2$-Grassmannian
    $G_2/P\subset \PP^{13}$.
    In each of these cases other than $\PP^4$, we will
    disprove Bott vanishing by a Riemann-Roch
    calculation.
    \begin{table}
  \centering
  \begin{tabular}{c|c|c|p{0.55\textwidth}}
    Fano index & $\sO(1)^4$ & $\chi(X,TX)$ & Description \\
    \hline
    2 & 2 & $-185$ & sextic in $\PP^5(1^53)$ \\
    2 & 4 & $-90$ & quartic in $\PP^5$ \\
    2 & 6 & $-55$ & complete intersection of quadric and cubic in $\PP^6$ \\
    2 & 8 & $-36$ & complete intersection of three quadrics in $\PP^7$ \\
    2 & 10 & $-24$ & $\Gr(2,5)\cap Q\cap \PP^7\subset \PP^9$ \\
    2 & 12 & $-15$ & $\OGr(5,10)^{+}\cap \PP^9\subset \PP^{15}$ \\
    2 & 14 & $-9$ & $\Gr(2,6)\cap \PP^{10}\subset \PP^{14}$ \\
    2 & 16 & $-3$ & $\LGr(3,6)\cap \PP^{11}\subset \PP^{13}$ \\
    2 & 18 & 1 & $G_2/P\cap \PP^{12}\subset \PP^{13}$ \\
    3 & 1 & $-103$ & sextic in $\PP^5(1^423)$ \\
    3 & 2 & $-45$ & quartic in $\PP^5(1^52)$ \\
    3 & 3 & $-20$ & cubic in $\PP^5$ \\
    3 & 4 & $-4$ & complete intersection of two quadrics in $\PP^6$ \\
    3 & 5 & 8 & $\Gr(2,5)\cap \PP^7\subset \PP^9$ \\
    4 & 2 & 15 & quadric $Q\subset\PP^5$ \\
    5 & 1 & 24 & $\PP^4$ \\
  \end{tabular}
  \caption{The Fano 4-folds of Picard number 1 with Fano index $\geq 2$}
  \label{fourfold}
\end{table}
    
    We know that Bott vanishing
    fails for the quadric fourfold, as mentioned
    in Proposition \ref{prop:Bott on hypersurface}.
    It remains to disprove Bott vanishing
    for the other two fourfolds above.
    First, let $Y$ be $G_2/P\subset \PP^{13}$;
    then $Y$ has dimension 5 
    and $-K_Y=\sO_Y(3)$. Using the Borel-Weil-Bott
    theorem, Konno showed (in characteristic
    zero, as here) that
    $TY(-a)$ has zero cohomology
    in all degrees for $1\leq a\leq 2$
    \cite[Theorem 3.4.1]{Konno}. By Lemma
    \ref{lem:hyperplane}, the Fano fourfold
    $X=G_2/P\cap \PP^{12}\subset \PP^{13}$
    has $\chi(X,TX(-1))=0-1=-1<0$. By adjunction,
    we have $-K_X=\sO(2)$, so
    $\chi(X,\Omega^3(1))=\chi(X,TX(-1))=-1<0$,
    and so $X$ does not
    satisfy Bott vanishing.

    Finally, let $Y$ be $\Gr(2,5)\subset \PP^9$,
    which has dimension 6 and Fano index 5.
    By applying Lemma \ref{lem:hyperplane} twice,
    a codimension-2
    linear section $X$ has
    \[
    \chi(X,TX(-1))=\chi(Y,TY(-1))-2\chi(Y,TY(-2))
    +\chi(Y,TY(-3))-2.
    \]
    Snow showed (in characteristic zero, as here)
    that if $Z$ is a Grassmannian $\Gr(s,t)$
    other than projective space or $\Gr(2,4)$,
    then $-K_Z=\sO_Z(t)$ and $TZ(-a)$ has zero cohomology
    in all degrees for $1\leq a\leq t-1$
    \cite[Theorem 3.4(3)]{Snow-Gr}.
    In particular, $Y=\Gr(2,5)$ has $\chi(Y,TY(-a))=0$
    for $1\leq a\leq 4$. By the formula above,
    the Fano fourfold $X=\Gr(2,5)\cap \PP^7\subset \PP^9$
    has $\chi(X,TX(-1))=-2<0$. By adjunction,
    $-K_X=\sO_X(3)$,
    So $\chi(X,\Omega^3(2))=\chi(X,TX(-1))=-2<0$,
    and so $X$ does not
    satisfy Bott vanishing. This completes the proof
    that if a smooth Fano fourfold of Picard number 1
    and Fano index greater than 1 satisfies Bott vanishing,
    then it is isomorphic to projective space.
\end{proof}

\section{Proof of Theorem \ref{Introthm:rank one}
and Theorem \ref{Introthm:degree}}

\begin{proof}[Proof of Theorems \ref{Introthm:rank one} and \ref{Introthm:degree}]
     In the situation of Theorem \ref{Introthm:rank one} (resp.~\ref{Introthm:degree}), 
     $X$ satisfies
     Bott vanishing by Theorem \ref{Introtheo:Bott}
     (resp.~Proposition \ref{prop:degree bound}),
     and the assertion follows from Propositions
     \ref{prop:Bott on hypersurface}, \ref{prop:rank one},
     and \ref{prop:index two}.
\end{proof}

\section{Failure of Bott vanishing
for separable polarized endomorphisms}
We now show that Bott vanishing can fail
if we only assume that
an int-amplified endomorphism
is separable (rather than of degree invertible in $k$).
This resolves a question in the first version of this paper.
(Bott vanishing obviously fails
for inseparable endomorphisms, since {\it every }projective
variety over a finite
field has the Frobenius endomorphism, which is int-amplified.)

\begin{prop}
\label{prop:separable}
For any prime power $q$ at least 4,
there is a smooth projective
3-fold $X$ over $\FF_q$ such that $X$ has a separable
polarized (hence int-amplified) endomorphism,
but Bott vanishing fails on $X$.
\end{prop}

\begin{proof}
Let $q$ be a prime power at least 4,
and let $X$ be the blow-up of $\PP^3$ over $\FF_q$
at some set of $\FF_q$-points on the plane
$\{w=0\}$. Then $X$
has a separable polarized endomorphism
(inspired by a 2-dimensional
example by Nakayama
\cite[Example 4.5]{Nakayama(endomorphism)}). Namely,
consider the endomorphism of $\PP^3$ given by
$$g([x,y,z,w])=[x^q-xw^{q-1},y^q-yw^{q-1},z^q-zw^{q-1},w^q].$$
Then $g$ is separable,
but it restricts to the Frobenius morphism
on the plane $\{w=0\}$. In particular,
$g$ fixes the given set of {$\FF_q$-}points
in the plane $\{w=0\}$. A direct calculation shows that $g$
lifts to an endomorphism $f$ of the blow-up $X$.
(It suffices to check this over
the point $[1,0,0,0]$ in $\PP^3$, in view of the symmetry
group $GL(3,\FF_q)$ (acting on $x,y,z$)
of the endomorphism $g$.)
Clearly $f$ is separable, since $g$ is.

We now specialize to the case where $X$ is the blow-up
of $\PP^3$ over $\FF_q$ at 5 $\FF_q$-points
in the plane $\{w=0\}$ with no 3 collinear. (This is
possible for $q\geq 4$, as we assumed.) These
points are contained in a unique smooth conic $C$. After
a change of coordinates in $GL(3,\FF_q)$, we can assume
that $C$ is the conic $\{w=0,xy=z^2\}$.
Write $H$ for the pullback to $X$
of the line bundle $\sO_{\PP^3}(1)$
on $\PP^3$,
and $E_1,\ldots,E_5$ for the exceptional divisors. Then
$f^*H=qH$ and $f^*E_j=qE_j$ for $j=1,\ldots,5$. Since $\Pic(X)
=\Z\{H,E_1,\ldots,E_5\}$, $f$ is polarized,
hence int-amplified.

Next, we show that the line bundle
$A\coloneqq 3H-\sum_{j=1}^5E_j$
is ample on $X$.
Indeed, $xw,yw,zw,w^2$, and $xy-z^2$ are sections
of $L\coloneqq 2H-\sum E_j$,
and so the base locus of $L$ is only the strict transform
of the conic $C$ in $\PP^3$. Also, $L$ has positive degree
on every curve
in $E_1,\ldots,E_5$. The base locus of $A=L+H$
is at most the conic $C$,
but the linear system $|A|$
also contains the sum of a plane through $p_1$ and $p_2$,
a plane through $p_1$ and $p_3$, and a plane through $p_4$ and $p_5$;
so $A$ is basepoint-free on $X$.
Since $A=L+H$ has $A\cdot C=1$, $A$ has
positive degree on every curve on $X$.
Together with basepoint-freeness,
this implies that $A$ is ample.

To disprove Bott vanishing on $X$, we will show that
$H^1(X,\Omega^1_X(A))$ is not zero.
Consider the exact sequence
$0\to \sO_X(-\sum_j E_j)\to \sO_X\to \oplus_j \sO_{E_j}\to 0$.
Tensoring with $\Omega^1_X(3H)$
and taking cohomology gives an exact sequence:
\[
H^0(X,\Omega^1_X(3H))\to
\oplus_{j=1}^5 H^0(E_j,\Omega^1_X(3H))
\to H^1(X,\Omega^1_X(A)).
\]
So it suffices to show that the restriction map
on $H^0$ is not
surjective. Since the line bundle $3H$ is pulled back
from $\PP^3$,
the first space is isomorphic
to $H^0(\PP^3,\Omega^1_{\PP^3}(3))$.
Next, each
$E_j$ is isomorphic to $\PP^2$,
and we have an exact sequence on $E\coloneqq E_j$:
$0\to \sO_X(-E)|_E\to \Omega^1_X|_E\to \Omega^1_E\to 0$, where
$\sO_X(-E)|_E=\sO_{E}(1)$ on $E=\PP^2$.
Here $H^0(E,\Omega^1_E)=0$,
and so $h^0(E,\Omega^1_X|_{E})=
h^0(E,\sO_{E}(1))=3$. The line bundle $H$ on $X$ is trivial on each
$E_j$, and so $h^0(E_j,\Omega^1_X(3H))=3$ for $j=1,\ldots,5$.
More canonically,
$H^0(E_j,\Omega^1_X(3H))\cong H^0(p_j,\Omega^1_{\PP^3}(3H))$.
So, to disprove Bott vanishing on $X$,
it suffices to show that the restriction map
$H^0(\PP^3,\Omega^1_{\PP^3}(3H))\to
\oplus_{j=1}^5 H^0(p_j,\Omega^1_{\PP^3}(3H))$
has rank less than $5\cdot 3=15$.

The point is that this restriction map factors through
$H^0(C,\Omega^1_{\PP^3}(3H))$, since $p_1,\ldots,p_5$
lie on the conic $C$.
To analyze that group,
note that the vector bundle $\Omega^1_{\PP^n}(2H)$
is globally generated for any $n$
\cite[equation 7.13]{LazarsfeldII}. As a result,
$\Omega^1_{\PP^3}(3H)$ is ample,
so its restriction
to $C\cong \PP^1$ is ample,
and hence $H^1(C,\Omega^1_{\PP^3}(3H))=0$. By Riemann-Roch,
it follows that
\begin{align*}
h^0(C,\Omega^1_{\PP^3}(3H))&=\chi(C,\Omega^1_{\PP^3}(3H))\\
&=\deg_C(\Omega^1_{\PP^3}(3H))
+\rank(\Omega^1_{\PP^3}(3H))(1-g(C))\\
&=10+3=13.
\end{align*}
So the restriction map
$H^0(\PP^3,\Omega^1_{\PP^3}(3H))\to
\oplus_{j=1}^5 H^0(p_j,\Omega^1_{\PP^3}(3H))$
has rank at most 13, thus less than 15. By the previous paragraph,
this completes the proof that Bott vanishing fails for $X$,
even though $X$ has a separable polarized endomorphism.
\end{proof}

\section{Global F-regularity of Fano varieties with an endomorphism}
Let $X$ be a Fano variety in characteristic $p>0$ that is
strongly $F$-regular (for example, smooth).
If $X$ admits an int-amplified
endomorphism of degree prime to $p$, we will show that $X$
is globally $F$-regular (Theorem \ref{Introtheo:global}).
(It was known to the experts
that a smooth Fano variety satisfying Bott vanishing
must be globally $F$-split,
by the argument sketched in \cite[Exercise 1.6.4]{fbook}. Therefore, when $X$ is smooth, Theorem \ref{Introtheo:global} is an immediate consequence of Theorem \ref{Introtheo:Bott}.)
Intuitively, ``strongly $F$-regular'' is a strong version
of ``klt type'' in characteristic $p$,
and ``globally $F$-regular'' is a strong version
of ``Fano type''.

\begin{defn}\label{defi-gFreg F-pure}
Let $X$ be a normal variety over a perfect field $k$
of characteristic $p>0$,
and let $B$ be an effective $\Q$-divisor on $X$.
\begin{enumerate}
\item
The pair $(X,B)$ is \textit{globally $F$-regular }if
for every effective Weil divisor $D$ on $X$,
there is a positive integer $e$ such that
the composite map
\[
\mathcal{O}_X \to F^e_*\mathcal{O}_X \hookrightarrow
F^e_*\mathcal{O}_X(\lceil (p^e-1)B\rceil + D)
\]
splits as an $\mathcal{O}_X$-module homomorphism
\cite[Definition 3.1]{SS10}. (Note that
$F^e_*\mathcal{O}_X(Z)$ means $F^e_*(\mathcal{O}_X(Z))$,
for a divisor $Z$.)
\item The pair $(X,B)$ is {\it globally
sharply $F$-split }if there
is a positive integer $e$ such that the composite map
\[
\mathcal{O}_X \to F^e_*\mathcal{O}_X \hookrightarrow
F^e_*\mathcal{O}_X(\lceil (p^e-1)B\rceil)
\]
splits as an $\mathcal{O}_X$-module homomorphism.
(For $B=0$, we omit the word ``sharply''.)
\item The pair $(X,B)$ is {\it strongly $F$-regular}, resp.\
{\it sharply $F$-pure}, if $X$ is covered
by open sets
on which the corresponding global property holds. 
(For $B=0$, we simply say {\it $F$-pure} to mean ``sharply $F$-pure.'')
\end{enumerate}
\end{defn}

To avoid confusion, note that whether
the map $\mathcal{O}_X\to F^e_*\mathcal{O}_X(Z)$ splits
depends on the effective divisor $Z$, not just
on its linear equivalence class.

Let $X$ be a smooth variety over a perfect field $k$
of characteristic $p>0$.
The Frobenius pushforward of the de Rham complex,
\[
F_{*}\Omega^{\bullet}_X : F_{*}\sO_X \overset{F_{*}d}{\to}
F_{*}\Omega_X \overset{F_{*}d}{\to} \cdots ,
\]
is a complex of $\sO_X$-module homomorphisms.
Define locally free $\sO_X$-modules as follows.
\[
\begin{array}{rl}
&B^i_X\coloneqq\Im(F_{*}d : F_{*}\Omega^{i-1}_X
\to F_{*}\Omega^i_X),\\
&Z^i_X\coloneqq\Ker(F_{*}d : F_{*}\Omega^{i}_X
\to F_{*}\Omega^{i+1}_X).\\
\end{array}
\]
By definition, we have an exact sequence
\begin{align}\label{closedfirst F-pure}
0 \to Z_X^{i} \to F_{*}\Omega_X^i \overset{F_{*}d}{\to}
B_X^{i+1}\to 0.
\end{align}
We also have the exact sequence
arising from the Cartier isomorphism
(see \cite[Theorem 1.3.4]{fbook}, for example),
\begin{align}\label{boundaryfirst F-pure}
0 \to B^i_X \to  Z^i_X \overset{C^i}{\to} \Omega^i_X \to 0.
\end{align}

\begin{thm}\label{thm:Bott implies GFR F-pure}
    Let $X$ be a Fano variety over a perfect field
    of characteristic $p>0$. Suppose that $X$ admits
    an int-amplified endomorphism of degree
    prime to $p$.
    \begin{itemize}
        \item[\textup{(1)}] If $X$ is strongly $F$-regular
        (for example, smooth),
        then it is globally $F$-regular.
        \item[\textup{(2)}] If $X$ is $F$-pure,
        then it is globally $F$-split.
    \end{itemize}
\end{thm}

\begin{rem}\label{rem:Fano-examples}
Theorem \ref{thm:Bott implies GFR F-pure} is sharp
in some ways. Consider the projective cone $X\subset \PP^3$
over a smooth cubic curve $C\subset \PP^2$
over an algebraically closed field $k$
of characteristic $p>0$. Then $X$ is a log canonical
Fano surface, and it admits an int-amplified
endomorphism of degree prime to $p$, coming
from a multiplication endomorphism of the
elliptic curve $C$. But $X$ is not strongly $F$-regular,
hence not globally $F$-regular. And if $C$
is supersingular, then $X$ is not $F$-pure, hence
not globally $F$-split. One might ask: is a klt Fano
variety with an int-amplified endomorphism
of degree prime to $p$ always globally $F$-regular?
\end{rem}

\begin{proof} (Theorem \ref{thm:Bott implies GFR F-pure})
    We first prove that global $F$-regularity
    of a Fano variety 
    is equivalent to global $F$-splitting plus
    strong $F$-regularity,
    using the results of Schwede and Smith.
    As a result, statement (2) will imply statement (1).

\begin{lem}\label{lem:integral}
Let $X$ be a normal quasi-projective variety
over a perfect field $k$,
and let $B$ be an effective $\Q$-divisor on $X$.
If the pair $(X,B)$
is globally sharply $F$-split, then there is an effective
$\Q$-divisor $\Delta$
such that $(X,B+\Delta)$ is globally sharply $F$-split, $B+\Delta$
has $\Z_{(p)}$ coefficients, and $K_X+B+\Delta$ is
$\Z_{(p)}$-linearly equivalent to zero.
\end{lem}

This is \cite[Theorem 4.3]{SS10}. They do not mention that
$K_X+B+\Delta$ is $\Z_{(p)}$-linearly equivalent to zero,
but that is what their proof gives (p.~878).

Let $X$ be a strongly $F$-regular Fano variety
that is globally $F$-split.
Then Lemma \ref{lem:integral}
gives an effective $\Z_{(p)}$-divisor $\Delta$ such that
$(X,\Delta)$ is globally $F$-split and $K_X+\Delta$
is $\Z_{(p)}$-linearly equivalent to zero.
By the definition of global $F$-splitting,
there is a positive integer $e$
such that $(p^e-1)\Delta$ has integer coefficients
and the inclusion $\sO_X\to F_*^e\sO_X((p^e-1)\Delta)$
is split. Here $\Delta$ is $\Z_{(p)}$-linearly equivalent
to $-K_X$, which is ample. So $X-\Supp(\Delta)$
is affine and strongly $F$-regular, hence globally
$F$-regular. By \cite[Theorem 3.9]{SS10},
it follows that $X$ is globally $F$-regular.

It remains to prove statement (2). That is, if $X$
is an $F$-pure Fano variety that admits
an int-amplified endomorphism of degree
prime to $p$, we will
show that $X$ is globally $F$-split.
Since $X$ is $F$-pure, the exact sequence
\[
0\to \sO_X\to F_*\sO_X\to F_*\sO_X/\sO_X\to 0
\]
is locally split on $X$. In particular, the sheaf
$F_*\sO_X/\sO_X$ is reflexive, since $F_*\sO_X$ is.
On the smooth locus $U$ of $X$, we have
$0\to \sO_U\to F_*\sO_U\to B^1_U\to 0$. So $F_*\sO_X/\sO_X$
is the double dual $B^{[1]}_X$ of $B^1_X$.

Since the sequence
\[
0\to \sO_X\to F_*\sO_X\to B^{[1]}_X\to 0
\]
is locally split, it corresponds to an element
of $H^1(X,\sHom(B^{[1]}_X, \sO_X))$.
We want to show that $X$ is globally
$F$-split, meaning that this element is zero.
We have a perfect
pairing $B^1_U\times B^d_U\to \omega_U$ on the smooth
locus $U$ of $X$ \cite[proof of Lemma 1.1]{Mehta-Srinivas}.
As a result, the sheaf $\sHom(B^{[1]}_X,\sO_X)$
is the reflexive sheaf $B^{[d]}_X(-K_X)$. So it suffices
to show that $H^1(X,B^{[d]}_X(-K_X))=0$. That follows
from Lemma \ref{lem:closed vanishing F-pure}, below. So $X$ is globally
$F$-split, proving statement (2).
Theorem \ref{thm:Bott implies GFR F-pure}
is proved.
\end{proof}

    \begin{lem}\label{lem:closed vanishing F-pure}
    Let $X$ be a normal projective variety
    over a perfect field of characteristic $p>0$.
    Suppose that $X$ admits an int-amplified endomorphism
    of degree prime to $p$. Then
    \[
    H^i(X,B_X^{[j]}(A))=0
    \]
    and
    \[
    H^i(X,Z_X^{[j]}(A))=0
    \]
    for every $i>0$, $j\geq 0$, and $A$ an ample Weil
    divisor.
    \end{lem}

    \begin{proof}
The proof of Theorem \ref{Introtheo:Bott} works without change
for the reflexive sheaves $B_X^{[j]}$ and $Z_X^{[j]}$
in place of $\Omega_X^{[j]}$.
In more detail, consider the pullback map
$\Omega^{[j]}_X\to f_*\Omega^{[j]}_X$ and
    the trace map $\tau_{f}
    \colon f_*\Omega^{[j]}_X\to \Omega^{[j]}_X$.
Because $f$ commutes with the Frobenius morphism $F$ on $X$,
we also have a pullback map
$F_*\Omega^{[j]}_X\to f_*F_*\Omega^{[j]}_X$ and
a trace map $\tau_{f}
    \colon f_*F_*\Omega^{[j]}_X\to F_*\Omega^{[j]}_X$.
We claim that these two maps preserve the subsheaves
$B_X^{[j]}$ and $Z_X^{[j]}$ of $F_*\Omega^{[j]}_X$;
then the proof
of Theorem \ref{Introtheo:Bott} applies.

Since $B_X^{[j]}$ and $Z_X^{[j]}$ are reflexive sheaves,
it suffices to check this claim
outside $X^{\sing}\cup f(X^{\sing})$. Then the claim follows
from the fact that the pullback and pushforward
of differential forms
commute with the exterior derivative $d$
\cite[Tag 0FLC]{stacks-project}.
\end{proof}

\section*{Acknowledgements}
We wish to express our gratitude to Shou Yoshikawa
for valuable conversations.
We are also grateful to Ekaterina Amerik,
Fabio Bernasconi, Frank Gounelas,
Masaru Nagaoka, Teppei Takamatsu,
De-Qi Zhang, and the referee for useful comments.
Kawakami was supported by JSPS KAKENHI Grant number
JP22KJ1771. Totaro was supported by NSF grant DMS-2054553.

\newcommand{\etalchar}[1]{$^{#1}$}


\begin{thebibliography}{ARVdV99}

\bibitem[Ame97]{Amerik-maps}
Ekaterina Amerik.
\newblock Maps onto certain {F}ano threefolds.
\newblock {\em Doc. Math.}, 2:195--211, 1997.

\bibitem[Ame07]{Amerik-quadric}
Ekaterina Amerik.
\newblock Mappings onto quadrics.
\newblock {\em Mat. Zametki}, 81(4):621--624, 2007.

\bibitem[Art74]{Artin-supersingular}
M.~Artin.
\newblock Supersingular {$K3$} surfaces.
\newblock {\em Ann. Sci. \'{E}cole Norm. Sup. (4)}, 7:543--567 (1975), 1974.

\bibitem[ARVdV99]{ARV}
Ekaterina Amerik, Marat Rovinsky, and Antonius Van~de Ven.
\newblock A boundedness theorem for morphisms between threefolds.
\newblock {\em Ann. Inst. Fourier (Grenoble)}, 49(2):405--415, 1999.

\bibitem[AWZ21]{AWZ}
Piotr Achinger, Jakub Witaszek, and Maciej Zdanowicz.
\newblock Global {F}robenius liftability {I}.
\newblock {\em J. Eur. Math. Soc. (JEMS)}, 23(8):2601--2648, 2021.

\bibitem[AWZ23]{AWZ2}
Piotr Achinger, Jakub Witaszek, and Maciej Zdanowicz.
\newblock Global {F}robenius liftability {II}: surfaces and {F}ano threefolds.
\newblock {\em Ann. Sc. Norm. Super. Pisa Cl. Sci.}, 24:329--366, 2023.

\bibitem[Bea01]{Beauville}
Arnaud Beauville.
\newblock Endomorphisms of hypersurfaces and other manifolds.
\newblock {\em Internat. Math. Res. Notices}, 2001(1):53--58, 2001.

\bibitem[BGI71]{SGA6}
Pierre Berthelot, Alexander Grothendieck, and Luc Illusie.
\newblock {\em Th\'{e}orie des intersections et th\'{e}or\`eme de
  {R}iemann-{R}och (SGA 6)}.
\newblock Lecture Notes in Mathematics, Vol. 225. Springer-Verlag, Berlin-New
  York, 1971.

\bibitem[Bha22]{Bhatt}
Bhargav Bhatt.
\newblock Prismatic {F}-gauges.
\newblock {\em Princeton lecture notes,\\
  https://www.math.ias.edu/\~{}bhatt/teaching/mat549f22/lectures.pdf}, 2022.

\bibitem[BK05]{fbook}
Michel Brion and Shrawan Kumar.
\newblock {\em Frobenius splitting methods in geometry and representation
  theory}, volume 231 of {\em Progress in Mathematics}.
\newblock Birkh\"{a}user Boston, Inc., Boston, MA, 2005.

\bibitem[BTLM97]{BTLM}
Anders Buch, Jesper~F. Thomsen, Niels Lauritzen, and Vikram Mehta.
\newblock The {F}robenius morphism on a toric variety.
\newblock {\em Tohoku Math. J. (2)}, 49(3):355--366, 1997.

\bibitem[CMZ20]{Cascini-Meng-Zhang}
Paolo Cascini, Sheng Meng, and De-Qi Zhang.
\newblock Polarized endomorphisms of normal projective threefolds in arbitrary
  characteristic.
\newblock {\em Math. Ann.}, 378(1-2):637--665, 2020.

\bibitem[dJ96]{deJong}
A.~J. de~Jong.
\newblock Smoothness, semi-stability and alterations.
\newblock {\em Inst. Hautes \'{E}tudes Sci. Publ. Math.}, 83:51--93, 1996.

\bibitem[Eke87]{Ekedahl}
Torsten Ekedahl.
\newblock Foliations and inseparable morphisms.
\newblock In {\em Algebraic geometry, {B}owdoin, 1985 ({B}runswick, {M}aine,
  1985)}, volume~46 of {\em Proc. Sympos. Pure Math.}, pages 139--149. Amer.
  Math. Soc., Providence, RI, 1987.

\bibitem[FGI{\etalchar{+}}05]{FAG}
Barbara Fantechi, Lothar G\"{o}ttsche, Luc Illusie, Steven~L. Kleiman, Nitin
  Nitsure, and Angelo Vistoli.
\newblock {\em Fundamental algebraic geometry}, volume 123 of {\em Mathematical
  Surveys and Monographs}.
\newblock American Mathematical Society, Providence, RI, 2005.

\bibitem[FS20]{Fanelli-unusual}
Andrea Fanelli and Stefan Schr\"{o}er.
\newblock The maximal unipotent finite quotient, unusual torsion in {F}ano
  threefolds, and exceptional {E}nriques surfaces.
\newblock {\em \'{E}pijournal G\'{e}om. Alg\'{e}brique}, 4:Art. 11, 29, 2020.

\bibitem[Fuj83]{Fujita}
Takao Fujita.
\newblock Vanishing theorems for semipositive line bundles.
\newblock In {\em Algebraic geometry ({T}okyo/{K}yoto, 1982)}, volume 1016 of
  {\em Lecture Notes in Math.}, pages 519--528. Springer, Berlin, 1983.

\bibitem[Fuj07]{Fujino-toric}
Osamu Fujino.
\newblock Multiplication maps and vanishing theorems for toric varieties.
\newblock {\em Math. Z.}, 257(3):631--641, 2007.

\bibitem[Gar84]{Garel}
Emmanuelle Garel.
\newblock An extension of the trace map.
\newblock {\em J. Pure Appl. Algebra}, 32(3):301--313, 1984.

\bibitem[GJ18]{Gounelas-Javanpeykar}
Frank Gounelas and Ariyan Javanpeykar.
\newblock Invariants of {F}ano varieties in families.
\newblock {\em Mosc. Math. J.}, 18(2):305--319, 2018.

\bibitem[Har77]{Har}
Robin Hartshorne.
\newblock {\em Algebraic geometry}.
\newblock Springer-Verlag, New York-Heidelberg, 1977.
\newblock Graduate Texts in Mathematics, No. 52.

\bibitem[HM03]{Jun-Muk--Mok}
Jun-Muk Hwang and Ngaiming Mok.
\newblock Finite morphisms onto {F}ano manifolds of {P}icard number 1 which
  have rational curves with trivial normal bundles.
\newblock {\em J. Algebraic Geom.}, 12(4):627--651, 2003.

\bibitem[IP99]{Algebraic-Geometry-V}
V.~A. Iskovskikh and Yu.~G. Prokhorov.
\newblock Fano varieties.
\newblock In {\em Algebraic geometry, {V}}, volume~47 of {\em Encyclopaedia
  Math. Sci.}, pages 1--247. Springer, Berlin, 1999.

\bibitem[Kah06]{Kahn}
Bruno Kahn.
\newblock Sur le groupe des classes d'un sch\'{e}ma arithm\'{e}tique.
\newblock {\em Bull. Soc. Math. France}, 134(3):395--415, 2006.
\newblock With an appendix by Marc Hindry.

\bibitem[Kaw21]{Kaw2}
Tatsuro Kawakami.
\newblock On {K}awamata-{V}iehweg type vanishing for three dimensional {M}ori
  fiber spaces in positive characteristic.
\newblock {\em Trans. Amer. Math. Soc.}, 374(8):5697--5717, 2021.

\bibitem[Kel22]{Keller}
Timo Keller.
\newblock On the {$p$}-torsion of the {T}ate-{S}hafarevich group of abelian
  varieties over higher dimensional bases over finite fields.
\newblock {\em J. Th\'{e}or. Nombres Bordeaux}, 34(2):497--513, 2022.

\bibitem[Kle05]{Kleiman}
Steven~L. Kleiman.
\newblock The {P}icard scheme.
\newblock In {\em Fundamental algebraic geometry}, volume 123 of {\em Math.
  Surveys Monogr.}, pages 235--321. Amer. Math. Soc., Providence, RI, 2005.

\bibitem[Kod86]{Kodaira-book}
Kunihiko Kodaira.
\newblock {\em Complex manifolds and deformation of complex structures}, volume
  283 of {\em Grundlehren der mathematischen Wissenschaften}.
\newblock Springer-Verlag, New York, 1986.
\newblock Translated from the Japanese by Kazuo Akao. With an appendix by
  Daisuke Fujiwara.

\bibitem[Kol13]{Kol13}
J\'{a}nos Koll\'{a}r.
\newblock {\em Singularities of the minimal model program}, volume 200 of {\em
  Cambridge Tracts in Mathematics}.
\newblock Cambridge University Press, Cambridge, 2013.
\newblock With a collaboration of S\'{a}ndor Kov\'{a}cs.

\bibitem[Kon89]{Konno}
Kazuhiro Konno.
\newblock Generic {T}orelli theorem for hypersurfaces of certain compact
  homogeneous {K}\"{a}hler manifolds.
\newblock {\em Duke Math. J.}, 59(1):83--160, 1989.

\bibitem[KP23]{Kuznetsov-Prokhorov}
A.~G. Kuznetsov and Yu.~G. Prokhorov.
\newblock On higher-dimensional del {P}ezzo varieties.
\newblock {\em Izv. Ross. Akad. Nauk Ser. Mat.}, 87(3):75--148, 2023.

\bibitem[Kun86]{Kunz}
Ernst Kunz.
\newblock {\em K\"{a}hler differentials}.
\newblock Advanced Lectures in Mathematics. Friedr. Vieweg \& Sohn,
  Braunschweig, 1986.

\bibitem[Laz84]{Lazarsfeld1983}
Robert Lazarsfeld.
\newblock Some applications of the theory of positive vector bundles.
\newblock In {\em Complete intersections ({A}cireale, 1983)}, volume 1092 of
  {\em Lecture Notes in Math.}, pages 29--61. Springer, Berlin, 1984.

\bibitem[Laz04]{LazarsfeldII}
Robert Lazarsfeld.
\newblock {\em Positivity in algebraic geometry. {II}}, volume~49 of {\em
  Ergebnisse der Mathematik und ihrer Grenzgebiete. 3. Folge. A Series of
  Modern Surveys in Mathematics [Results in Mathematics and Related Areas. 3rd
  Series. A Series of Modern Surveys in Mathematics]}.
\newblock Springer-Verlag, Berlin, 2004.
\newblock Positivity for vector bundles, and multiplier ideals.

\bibitem[Meg98]{Megyesi}
G.~Megyesi.
\newblock Fano threefolds in positive characteristic.
\newblock {\em J. Algebraic Geom.}, 7(2):207--218, 1998.

\bibitem[Men20]{Meng-building}
Sheng Meng.
\newblock Building blocks of amplified endomorphisms of normal projective
  varieties.
\newblock {\em Math. Z.}, 294(3-4):1727--1747, 2020.

\bibitem[Mil80]{Milne-etale}
James~S. Milne.
\newblock {\em \'{E}tale cohomology}.
\newblock Princeton Mathematical Series, No. 33. Princeton University Press,
  Princeton, N.J., 1980.

\bibitem[MP12]{MP12}
Davesh Maulik and Bjorn Poonen.
\newblock N\'{e}ron-{S}everi groups under specialization.
\newblock {\em Duke Math. J.}, 161(11):2167--2206, 2012.

\bibitem[MS87]{Mehta-Srinivas}
V.~B. Mehta and V.~Srinivas.
\newblock Varieties in positive characteristic with trivial tangent bundle.
\newblock {\em Compositio Math.}, 64(2):191--212, 1987.
\newblock With an appendix by Srinivas and M. V. Nori.

\bibitem[Muk89]{Mukai-Fano}
Shigeru Mukai.
\newblock Biregular classification of {F}ano {$3$}-folds and {F}ano manifolds
  of coindex {$3$}.
\newblock {\em Proc. Nat. Acad. Sci. U.S.A.}, 86(9):3000--3002, 1989.

\bibitem[MZ20]{Meng-Zhang-normal}
Sheng Meng and De-Qi Zhang.
\newblock Normal projective varieties admitting polarized or int-amplified
  endomorphisms.
\newblock {\em Acta Math. Vietnam.}, 45(1):11--26, 2020.

\bibitem[MZZ22]{Meng-Zhang-Zhong}
Sheng Meng, De-Qi Zhang, and Guolei Zhong.
\newblock Non-isomorphic endomorphisms of {F}ano threefolds.
\newblock {\em Math. Ann.}, 383(3-4):1567--1596, 2022.

\bibitem[Nak10]{Nakayama(endomorphism)}
Noboru Nakayama.
\newblock Separable endomorphisms of surfaces in positive characteristic.
\newblock In {\em Algebraic geometry in {E}ast {A}sia---{S}eoul 2008},
  volume~60 of {\em Adv. Stud. Pure Math.}, pages 301--330. Math. Soc. Japan,
  Tokyo, 2010.

\bibitem[OW02]{Occhetta-Wisniewski}
Gianluca Occhetta and Jaros\l aw~A. Wi\'{s}niewski.
\newblock On {E}uler-{J}aczewski sequence and {R}emmert-van de {V}en problem
  for toric varieties.
\newblock {\em Math. Z.}, 241(1):35--44, 2002.

\bibitem[PS89]{Paranjape-Srinivas}
K.~H. Paranjape and V.~Srinivas.
\newblock Self-maps of homogeneous spaces.
\newblock {\em Invent. Math.}, 98(2):425--444, 1989.

\bibitem[SB97]{SB97}
N.~I. Shepherd-Barron.
\newblock Fano threefolds in positive characteristic.
\newblock {\em Compositio Math.}, 105(3):237--265, 1997.

\bibitem[Sno86]{Snow-Gr}
Dennis~M. Snow.
\newblock Cohomology of twisted holomorphic forms on {G}rassmann manifolds and
  quadric hypersurfaces.
\newblock {\em Math. Ann.}, 276(1):159--176, 1986.

\bibitem[SPA23]{stacks-project}
The Stacks Project~Authors.
\newblock The {S}tacks {P}roject.
\newblock \url{https://stacks.math.columbia.edu/}, 2023.

\bibitem[SS10]{SS10}
Karl Schwede and Karen~E. Smith.
\newblock Globally {$F$}-regular and log {F}ano varieties.
\newblock {\em Adv. Math.}, 224(3):863--894, 2010.

\bibitem[SZ24]{Shao-Zhong}
Feng Shao and Guolei Zhong.
\newblock Boundedness of finite morphisms onto {F}ano manifolds with large
  {F}ano index.
\newblock {\em J. Algebra}, 639:678--707, 2024.

\bibitem[Tan22]{Tanaka-kv}
Hiromu Tanaka.
\newblock Kawamata-{V}iehweg vanishing for toric varieties.
\newblock {\em arXiv preprint arXiv:2208.09680}, 2022.

\bibitem[Tan23]{Tanaka-fano4}
Hiromu Tanaka.
\newblock Fano threefolds in positive characteristic {IV}.
\newblock {\em arXiv preprint arXiv:2308.08127}, 2023.

\bibitem[Tot23a]{Totaro}
Burt Totaro.
\newblock Bott vanishing for {F}ano $3$-folds.
\newblock {\em arXiv preprint arXiv:2302.08142, to appear in Math. Z.}, 2023.

\bibitem[Tot23b]{Totaro-endo}
Burt Totaro.
\newblock Endomorphisms of {F}ano 3-folds and log {B}ott vanishing.
\newblock {\em arXiv preprint arXiv:2305.18660}, 2023.

\bibitem[Wil87]{Wilson}
P.~M.~H. Wilson.
\newblock Fano fourfolds of index greater than one.
\newblock {\em J. Reine Angew. Math.}, 379:172--181, 1987.

\end{thebibliography}

\bigskip

\end{document}